\documentclass{amsart}

\usepackage{amsfonts,amsthm,amsmath,amssymb,latexsym}
\usepackage{graphicx,color}
\usepackage[all]{xy}

\newtheorem{theorem}{Theorem}[section]
\newtheorem{definition}{Definition}[section]
\newtheorem{lemma}{Lemma}[section]

\newtheorem{proposition}{Proposition}[section]
\newtheorem{corollary}{Corollary}[section]

\theoremstyle{remark}
\newtheorem*{remark}{Remark}

\def\H{{\bf H}}
\def\R{{\bf R}}
\def\D{{\bf D}}

\begin{document}

\author{Dragomir \v Sari\' c}

\address{Department of Mathematics, Queens College of CUNY,
65-30 Kissena Blvd., Flushing, NY 11367}
\email{Dragomir.Saric@qc.cuny.edu}

\title{Some remarks on bounded earthquakes}


\date{\today}

\maketitle

\begin{abstract}
We first show that an earthquake of a geometrically infinite
hyperbolic surface induces an asymptotically conformal change in the
hyperbolic metric if and only if the measured lamination associated
with the earthquake is asymptotically trivial on the surface. Then
we show that the contraction along earthquake paths is continuous in
the Teichm\"uller space of any hyperbolic surface. Finally, we show
that if a measured lamination vanishes while approaching infinity at
the rate higher than the distance to the boundary then it must be
trivial.
\end{abstract}

\section{Introduction}

An earthquake $E$ of a hyperbolic surface $M$ is obtained by
partitioning $M$ into the strata and by moving one stratum at a
time. The support of an earthquake $E$ is a geodesic lamination
$\lambda$ on $M$. Each stratum of $E$ is either a geodesic from
$\lambda$ or a component of the complement of $\lambda$ in $M$. The
earthquake $E$ does not change the shape of each stratum (i.e. $E$
is a hyperbolic isometry on each stratum) and it moves each stratum
to the left relative to any other stratum. For all the details see
Thurston \cite{Thu} or Section 2.

The relative movement to the left between any two strata of $E$ is
best described by a transverse measure to $\lambda$ which is
invariant under homotopies relative $\lambda$. A geodesic lamination
together with a transverse measure is called a measured lamination.
An earthquake $E$ is uniquely determined (up to post-composition
with a hyperbolic isometry) by its measured lamination. Any
earthquake $E$ of $M$ onto another hyperbolic surface $M_1$ induces
a homeomorphism between the boundaries of their corresponding
universal coverings (we take the upper half-plane $\H$ as universal
coverings), i.e. $E$ induces a homeomorphism of the extended real
line $\hat{\R}$ which conjugates the uniformizing Fuchsian group of
$M$ onto the uniformizing Fuchsian group of $M_1$. Thurston
\cite{Thu} showed that any homeomorphism of $\hat{\R}$ is induced by
an earthquake $E$ of the upper half-plane $\H$, and if a
homeomorphism conjugates a Fuchsian group $G$ onto another Fuchsian
group then $E$ is an earthquake of $\H /G$.

An earthquake $E$ of the upper half-plane $\H$ induces a
quasisymmetric homeomorphism of $\hat{\R}$ if and only if the
transverse measure of $E$ is uniformly bounded on all hyperbolic
arcs of length $1$ (see \cite{Sa1}, \cite{GHL}, \cite{Hu},
\cite{EMM}, \cite{Sa2} and \cite{Thu}, or Section 2). Moreover, an
earthquake $E$ of $\H$ induces a symmetric homeomorphism of
$\hat{\R}$ if and only if the transverse measure of $E$ tends to
zero uniformly on hyperbolic arcs of length $1$ which converge to
the boundary $\hat{\R}$ of $\H$ (see \cite{GHL}, \cite{Hu} and
\cite{Sa2}, or Section 2).

A symmetric homeomorphism of $\hat{\R}$ extends to an asymptotically
conformal quasiconformal map of $\H$ and an asymptotically conformal
quasiconformal map of $\H$ extends to a symmetric homeomorphism of
$\hat{\R}$ (for details of the definitions see Section 2). A proper
generalization of the space of all symmetric homeomorphisms of
$\hat{\R}$ for geometrically infinite hyperbolic surface $M$ is the
space of all asymptotically conformal quasiconformal maps from $M$
onto variable hyperbolic surfaces $M_1$ up to bounded homotopy (see
\cite{GS}, \cite{EGL}). A lift to $\H$ of an asymptotically
conformal quasiconformal map of $M$ does not induce a symmetric
homeomorphism of $\hat{\R}$ (unless $M=\H $). Our first result is a
characterization of earthquakes of $M$ whose lifts to $\H$ agree on
the boundary $\hat{\R}$ with the continuous extensions of lifts of
asymptotically conformal quasiconformal maps of $M$ (see Theorem
3.1).

Let $M$ be a geometrically infinite hyperbolic surface and let $E$
be an earthquake of $M$. Let $\tilde{E}:\H\to \H$ be a lift of $E$
to the universal covering and let $\tilde{E}|_{\hat{\R}}$ be its
continuous extension to the boundary $\hat{\R}$. Let $f:M\to M_1$ be
a quasiconformal map. Denote by $\tilde{f}:\H\to\H$ a lift of $f$ to
the universal covering $\H$ and by $\tilde{f}|_{\hat{\R}}$ the
continuous extension to $\hat{\R}$. The measured lamination
associated with $E$ is said to be {\it asymptotically trivial} on
$M$ if the supremum of the measures of hyperbolic arcs of length $1$
outside a compact subset $K$ of $M$ is approaching zero as $K$
exhausts $M$.

\vskip .2 cm

\noindent {\bf Theorem 1.} {\it An earthquake $E$ of a geometrically
infinite hyperbolic surface $M$ induces a homeomorphism
$\tilde{E}|_{\hat{\R}}$ of $\hat{\R}$ which agrees with an induced
homeomorphism $\tilde{f}|_{\hat{\R}}$ of an asymptotically conformal
quasiconformal map $f:M\to M_1$ if and only if the transverse
measure of $E$ is asymptotically trivial on $M$.}

\vskip .2 cm

Let $ML_b(\H )$ denote the space of all bounded measured laminations
on $\H$. Each quasisymmetric homeomorphism $h:\hat{\R}\to \hat{\R}$
is obtained by continuously extending to $\hat{\R}$ an earthquake
$E$ of $\H$ with associated bounded measured lamination. This gives
an {\it earthquake measure} map $\mathcal{EM}:T(\H )\to ML_b(\H )$,
where $T(\H )$ is the universal Teichm\"uller space which can be
identified with the space of all quasisymmetric maps of $\hat{\R}$
modulo post-composition with $PSL_2(\R )$. If $\H$ is replaced with
a compact hyperbolic surface $S$, then the weak* topology on $ML(S)$
makes the earthquake measure map $\mathcal{EM}:T(S )\to ML(S )$ a
homeomorphism. The question about a natural topology on $ML_b(\H )$
which makes $\mathcal{EM}:T(\H )\to ML_b(\H )$ into homeomorphism is
still open (by a natural topology on $ML_b(\H )$ we mean a topology
given in terms of $\H$ and measured laminations without reference to
the earthquake measure map). However, we show that the earthquake
measure map is natural for the multiplication of the measured
laminations by a positive parameter (see Theorem 4.1).

\vskip .2 cm

\noindent {\bf Theorem 2.} {\it The scaling measure map
$H:[0,1]\times T(\H )\to T(\H )$ given by
$$
H(t,[h])=[E^{(1-t)\mathcal{EM}([h])}|_{\hat{\R}}]
$$
is continuous.}

\vskip .2 cm

The above theorem is true for Teichm\"uller spaces of arbitrary
hyperbolic surface $M$ as well. A direct corollary to the above
theorem is that the Teichm\"uller space $T(M)$ is contractible (also
the closed subspace of $T(M)$ of classes with asymptotically
conformal representatives is contractible). This was originally
proved by Douady and Earle \cite{DE}. We obtained a new proof of
this fact using earthquakes (see Corollary 4.1).

\vskip .2 cm

Finally, we consider measured laminations whose measures vanish as
they approach boundary. Gardiner, Hu and Lakic \cite{GHL} considered
earthquakes whose measured laminations $\mu$ satisfy $\mu
=o(\delta^{\alpha} )$, for $\alpha >0$. (By the definition, $\mu
=o(\delta^{\alpha} )$ if the supremum of $\mu (I)$ over all
hyperbolic arcs of length $1$ which are on the distance at most
$\delta$ from the boundary is $o(\delta^{\alpha} )$.) We show that
if $\alpha =1$, i.e. $\mu =o(\delta )$, then $\mu$ is a trivial
measured lamination, i.e. $\mu =0$ (see Proposition 5.1).

\section{Preliminaries: Bounded Earthquakes}

Let $\H$ be the upper half-plane with the metric $\rho
(z)=\frac{|dz|}{y}$. Denote by $\hat{\R}=\R\cup\{\infty\}$ the
boundary at infinity of $\H$. We identify the hyperbolic plane
with the model $(\H ,\rho )$.

Let $M$ be a Riemann surface of infinite geometric type. The
universal covering of $M$ is the hyperbolic plane $\H$. The
Riemann surface $M$ supports a unique hyperbolic metric in its
conformal class which is the projection of $\rho$ under the
universal covering by $\H$. Let $G$ be a subgroup of $PSL_2(\R )$
such that $M$ is conformal and isometric to $\H /G$. From now on,
we fix one such $G$ and the identification $M=\H /G$.

\begin{definition}
A geodesic lamination $\lambda$ on $\H$ is a closed subset $|\lambda
|$ of $\H$ together with an assignment of a foliation of $|\lambda
|$ by geodesics of $\H$. A geodesic lamination $\lambda_G$ on $M=\H
/G$ is a closed subset $|\lambda_G|$ of $\H$ which is foliated by
geodesics of $\H$ such that the subset $|\lambda_G|$ and the
foliation of $|\lambda_G|$ are invariant under $G$. A stratum of
$\lambda$ is either a geodesic of the foliation of $|\lambda |$ or a
connected component of $\H\setminus |\lambda |$.
\end{definition}

\begin{remark}
A component of $\H\setminus |\lambda |$ is an open hyperbolic
polygon whose boundary in $\H$ consists of geodesics of the
foliation. The polygon is possibly infinite-sided and it can contain
open intervals of $\hat{\R}$ on its boundary at infinity.
\end{remark}

\begin{definition}
Let $\lambda$ be a geodesic lamination on $\H$. A measured
lamination $\mu$ on $\H$ with support $\lambda$ is an assignment of
a positive, Radon measure to each closed, finite hyperbolic arc $I$
in $\H$ whose support is $I\cap |\lambda |$ and which is invariant
under homotopies which preserve the foliation of $|\lambda |$. (We
do not require that the homotopies preserve the endpoints of $I$.) A
measured lamination on $\H /G$ is an assignment of a measured
lamination $\mu_G$ on $\H$ such that the support $\lambda_G$ is a
measured lamination on $\H /G$ (i.e. $\lambda_G$ is invariant under
$G$) and that the measure $\mu_G$ is invariant under the action of
$G$.
\end{definition}

We define an earthquake of the hyperbolic plane $\H$ following Thurston \cite{Thu}.

\begin{definition}
An earthquake $E$ of $\H$ whose support is a geodesic lamination
$\lambda$ is a bijective map $E:\H\to\H$ such that the restriction
of $E$ to any stratum of $\lambda$ is in $PSL_2(\R )$ and that $E$
moves strata of $\lambda$ to the left relatively to each other.
Namely, for any stratum $g$ of $\lambda$ we have that $E|_g\in
PSL_2(\R )$, and for any two strata $g_1$, $g_2$ of $\lambda$ we
have that $E|_{g_2}\circ (E|_{g_1})^{-1}$ is a hyperbolic isometry
whose axis is weakly separating the two strata $g_1$ and $g_2$, and
which translates $g_2$ to the left as seen from $g_1$. An earthquake
$E$ of $\H$ is said to be an earthquake of $\H /G$ if the support
$\lambda$ is invariant under $G$, and for every $\gamma\in G$ there
exists $\gamma_E\in PSL_2(\R )$ such that $E\circ\gamma
=\gamma_E\circ E$.
\end{definition}

\vskip .3 cm

An earthquake $E$ of $\H$ naturally induces a measured lamination
$\mu$ with support $\lambda$ as follows. The measure $\mu$ on each
closed, hyperbolic arc $I$ is approximated by the sum of the
translation lengths between the comparison isometries of $n$
consecutive strata of the support $\lambda$ of $E$ which intersect
$I$ such that the distance between any two consecutive strata is
small for $n$ large (see \cite{Thu}, \cite{GHL}). The measure $\mu
(I)$ is obtained by taking the limit of the sum of translation
lengths as $n\to\infty$. It is clear that $\mu$ is homotopy
invariant. An earthquake $E$ is uniquely determined by its measured
lamination up to post-composition with $PSL_2(\R )$ \cite{Thu}. A
measured lamination corresponding to an earthquake is called an {\it
earthquake measure}. Moreover, $E$ is an earthquake on $\H /G$ if
and only if the earthquake measure of $E$ is a measured lamination
on $\H /G$.

\vskip .3 cm

An earthquake of $\H$ is not necessarily a homeomorphism of $\H$. A
discontinuity of an earthquake appears at the support geodesics of
the earthquake measure $\mu$ which are atoms of $\mu$. However, an
earthquake always extends by continuity to a homeomorphism of the
boundary $\hat{\R}$ of $\H$. Thurston \cite{Thu} proved that any
homeomorphism of $\hat{\R}$ can be obtained as the continuous
extension of an  earthquake of $\H$ to the boundary $\hat{\R}$. In
addition, any homeomorphic deformation of the hyperbolic structure
on $M=\H /G$ is obtained by an earthquake of $M$ \cite{Thu}.

\begin{definition}
A measured lamination $\mu$ with the support $\lambda$ is said to be bounded if
$$
\|\mu\| =\sup_I\mu (I)<\infty ,
$$
where the supremum is over all length $1$ closed hyperbolic arcs $I$ in $\H$.
\end{definition}

An earthquake $E$ of $\H$ extends by continuity to a quasisymmetric map of $\hat{\R}$ if and only if
the earthquake measure $\mu$ of $E$ is bounded (see \cite{Sa1}, \cite{GHL}, \cite{Thu}, \cite{EMM}, \cite{Sa2}).

\begin{definition}
A measured lamination $\mu$ of $\H$ is said to be asymptotically trivial if for every $\epsilon >0$ there
exists a compact subset $K\subset\H$ such that
$$
\sup_{I\cap K=\emptyset}\mu (I)<\epsilon ,
$$
where the supremum is over all length $1$ hyperbolic arcs which do not intersect $K$.
\end{definition}

A quasisymmetric homeomorphism of $\hat{\R}$ is said to be symmetric
if it distorts symmetric triples into almost symmetric triples on
small scales \cite{GS}. Equivalently, a quasisymmetric homeomorphism
$h:\hat{\R}\to\hat{\R}$ is symmetric if there exists a
quasiconformal extension $f:\H\to\H$ of $h$ such that for every
$\epsilon >0$ there exists a compact set $K\subset\H$ with $\|Belt
(f)|_{\H\setminus K}\|_{\infty}<\epsilon$, where $Belt
(f)=\bar{\partial}f/\partial f$ is the Beltrami coefficient of $f$.

It was proved by Gardiner, Hu and Lakic \cite{GHL} that an
earthquake $E$ extends to a symmetric homeomorphism of $\hat{\R}$
if and only if the earthquake measure $\mu$ of $E$ is
asymptotically trivial (see also \cite{Hu} and \cite{Sa2} for
different proofs).

\section{Asymptotically conformal structures on infinite surfaces via earthquakes}

Let $M=\H /G$ be a geometrically infinite hyperbolic surface. A quasiconformal map $F:M\to M_1$ is
said to be {\it asymptotically conformal} on $\H /G$ if for any $\epsilon >0$ there exists a compact set
$K\subset M$ such that $\| Belt(F) |_{M\setminus K}\|_{\infty}<\epsilon$, where
$Belt(F) =\bar{\partial}F/\partial F$ is the Beltrami coefficient of $F$.

\begin{definition}
A measured lamination $\mu_G$ on $\H /G$ is said to be {\it asymptotically trivial} on $\H /G$ if
for any $\epsilon >0$ there exists a compact subset $K$ of $\H /G$ such that
$\sup_{I\cap\tilde{K}=\emptyset}\mu (I)<\epsilon$, where the supremum is over all closed hyperbolic arcs
$I$ of length $1$ which do not intersect the lift $\tilde{K}$ of $K$ to $\H$.
\end{definition}

\begin{theorem}
Let $F:M\to M_1$ be a quasiconformal map. Fix a lift $f:\H\to\H$
of $F$ which conjugates $G$ onto $G_1$ and denote by
$h:\hat{\R}\to\hat{\R}$ the continuous extension of $f$ to
$\hat{\R}$. Let $E:\H\to\H$ be the earthquake on $\H /G$ whose
continuous extension to $\hat{\R}$ equals $h$. Then the earthquake
measure $\mu_h$ of $E$ is asymptotically trivial on $\H /G$ if and
only if $F$ is homotopic to an asymptotically conformal map on $\H
/G$ through a bounded homotopy.
\end{theorem}

\begin{proof}
Let $F:M\to M_1$ be homotopic to an asymptotically conformal map
through a bounded homotopy. Let $DE(h)$ be the Douady-Earle
extension of the quasisymmetric map $h:\hat{\R}\to\hat{\R}$ to $\H$
(see \cite{DE}). Then $DE(h)$ is necessarily asymptotically
conformal on $\H /G$ (see \cite{EMS}), namely for every $\epsilon
>0$ there exists a compact set $K\subset \H /G$ such that $\| Belt
(DE(h))|_{\H\setminus\tilde{K}}\|_{\infty}<\epsilon$, where
$\tilde{K}$ is the lift of $K$ to $\H$. Assume on the contrary that
the earthquake measure $\mu_h$ of $E$ is not asymptotically trivial
on $M=\H /G$. Thus, there exist $\epsilon_0>0$ and a sequence $I_n$
of length $1$ closed hyperbolic arcs in $\H$ leaving the preimage in
$\H$ of any compact subset of $M$ such that
$\mu_h(I_n)\geq\epsilon_0$. Let $\gamma_n\in PSL_2(\R )$ be such
that $\gamma_n(I_n)\ni i$. Let $\mu_n=(\gamma_n)^{*}(\mu_h)$ be the
push-forward of $\mu_h$ by $\gamma_n$. Note that the support of
$\mu_n$ is $\gamma_n(\lambda_h)$, where $\lambda_h$ is the support
of $\mu_h$. Then $\mu_n$ is the earthquake measure of an earthquake
$E^{\mu_n}$ which agrees with $h\circ \gamma_n^{-1}$ on $\hat{\R}$
up to post-composition with an element of $PSL_2(\R )$. Since
$\|\mu_n\| =\|\mu_h\|$ for each $n$, it follows that there is a
subsequence $\mu_{n_k}$ of the sequence $\mu_n$ which converges in
the weak* sense to a measured lamination $\mu_{*}$ with $\|\mu_{*}\|
\leq \|\mu_h\|$. Without loss of generality, we assume that the
whole sequence $\mu_n$ converges. This implies that a properly
normalized sequence of earthquakes $E^{\mu_{n}}|_{\hat{\R}}$
converge weakly to $E^{\mu_{*}}|_{\hat{\R}}$, where $E^{\mu_{n}}$
and $E^{\mu_{*}}$ are earthquakes with earthquake measures $\mu_{n}$
and $\mu_{*}$ respectively (see \cite[Proposition 3.3]{Sa2}). Let
$\delta_{n}\in PSL_2(\R )$ be such that $\delta_{n\frac{}{}}\circ
h\circ\gamma_n^{-1}|_{\hat{\R}}=E^{\mu_n}|_{\hat{\R}}$ for all $n$.
Since $I_n$ leaves the preimage of every compact subset of $\H /G$
and by the conformal naturality of the Douady-Earle extension, it
follows that the Beltrami coefficient $Belt(DE(\delta_n\circ
h\circ\gamma_n^{-1}))$ converges to zero uniformly on compact
subsets of $\H$. In addition, since $\delta_n\circ
h\circ\gamma_n^{-1}$ weakly converges, it follows that
$DE(\delta_n\circ h\circ\gamma_n^{-1})$ converges to a M\"obius map
by the properties of the Douady-Earle extension (see \cite{DE} and
\cite{EMS}). On the other hand, $\delta_n\circ
h\circ\gamma_n^{-1}=E^{\mu_n}|_{\hat{\R}}$ weakly converges to
$E^{\mu_{*}}|_{\hat{\R}}$ with $\mu_{*}$ non-trivial. This implies
that $DE(E^{\mu_{*}}|_{\hat{\R}})$ is not M\"obius which is a
contradiction.

\vskip .3 cm

Let $\mu_h$ be an asymptotically trivial earthquake measure on $\H
/G$ corresponding to $h$. We need to show that $DE(h)$ is
asymptotically conformal. Assume on the contrary that there exists a
sequence $z_n\in\H$ which leaves the preimage in $\H$ of every
compact subset of $\H /G$ such that
$|Belt(DE(h))(z_n)|\geq\epsilon_0$, for some fixed $\epsilon_0>0$.
Let $\gamma_n\in PSL_2(\R )$ be such that $\gamma_n(z_n)=i$. Let
$\mu_n=(\gamma_n)^{*}\mu_h$. Then there exists a subsequence
$\mu_{n_k}$ of $\mu_n$ which converges in the weak* sense to a
measured lamination $\mu_{*}$ with $\|\mu_{*}\|\leq\|\mu_h\|$.
Without loss of generality we assume that the sequence $\mu_n$
converges. By the choice of $\gamma_n$, it follows that $\mu_{*}$ is
the trivial measured lamination, i.e. $\mu_{*}=0$. This implies that
there exists $\delta_n\in PSL_2(\R )$ such that $\delta_n\circ
h\circ\gamma_n^{-1}=E^{\mu_n}|_{\hat{\R}}$ converges weakly to the
identity, where $E^{\mu_n}|_{\hat{\R}}$ is a properly normalized
earthquake \cite{Sa2}. Then $Belt(DE(\delta_n\circ
h\circ\gamma_n^{-1}))$ converges uniformly to $0$ on compact subsets
of $\H$. However, $|Belt(DE(\delta_n\circ
h\circ\gamma_n^{-1}))(i)|=|Belt(DE(h))(z_n)|\geq\epsilon_0>0$ which
is a contradiction.
\end{proof}

\section{The earthquake measure map and the contractibility}

Given a quasisymmetric map $h:\hat{\R}\to\hat{\R}$, there is a
unique earthquake $E_h$ of $\H$ whose continuous extension to
$\hat{\R}$ coincides with $h$. Let $\mu_h$ be the earthquake measure
of $E_h$ ($\mu_h$ is bounded \cite{Sa2}). Thus we obtain a one to
one assignment $[h]=PSL_2(\R )\circ h\mapsto \mu_h$ from the
universal Teichm\"uller space $T(\H )$ onto the space of bounded
measured laminations $ML_b(\H )$ of $\H$. (Note that $[h]$ denotes
the class $PSL_2(\R )\circ h$ of a quasisymmetric map $h$ and that
$T(\H )$ is the space of all such classes of quasisymmetric maps.)
We denote this map by $\mathcal{EM}:T(\H )\to ML_b(\H )$ and call it
the {\it earthquake measure map}. The earthquake measure map
restricts to a map from the Teichm\"uller space of a hyperbolic
Riemann surface $M$ onto the space of bounded measured laminations
on $M$ (see \cite{Thu}).

If $M$ is a compact surface of genus at least two, then the weak*
topology on $ML_b(M)$ makes the earthquake measure map
$\mathcal{EM}:T(M)\to ML_b(M)$ a homeomorphism (see \cite{Ke}). If
$M$ is a geometrically infinite Riemann surface (including the
possibility that $M=\H$), then it does not appear that there is a
natural topology on $ML_b(M)$ which would make $\mathcal{EM}:T(M)\to
ML_b(M)$ a homeomorphism. (One can induce a topology on $ML_b(M)$ by
the pull back of the topology on $T(M)$ but there is no information
about such assigned topology on $ML_b(M)$. A "natural" topology on
$ML_b(M)$ should be given in terms of the measured laminations on
the surface $M$ without the use of the quasisymmetric maps.)
However, we show below that the earthquake measure map
$\mathcal{EM}:T(M)\to ML_b(M)$ is natural for the scaling of the
earthquake measure.

\begin{theorem}
The scaling measure map $H:[0,1]\times T(\H )\to T(\H )$ given by
$$
H(t,[h])=[E^{(1-t)\mathcal{EM}([h])}|_{\hat{\R}}]
$$
is continuous.
\end{theorem}

\begin{remark}
The map $H:[0,1]\times T(\H )\to T(\H )$ restricts to a continuous
map $H:[0,1]\times T(\H /G)\to T(\H /G)$, for any Fuchsian group
$G$.
\end{remark}

\begin{proof}
Let $[h]\in T(\H )$. We show that $H$ is continuous at $(t,[h])$.
Let $[h_n]\in T(\H )$ be a sequence converging to $[h]$ in the
Teichm\"uller metric of $T(\H )$ and let $t_n\to t$ as
$n\to\infty$. We need to show that $H(t_n,[h_n])\to H(t,[h])$ as
$n\to\infty$.

We give a lemma which is used in the proof. Let
$\mathcal{G}=\hat{\R}\times\hat{\R}\setminus diag$ be the space of
(oriented) geodesics in $\H$. We define an angle metric on
$\hat{\R}$ as follows. The distance between two points
$x,y\in\hat{\R}$ is given by the minimum of the two angles between
hyperbolic rays both starting at $i$ and ending at $x$ and at $y$,
respectively. The angle metric induces a metric on $\mathcal{G}$.
Let $Q_{*}=[a,b]\times [c,d]\subset \mathcal{G}$, for $a,b,c,d$ in
counterclockwise order on $\hat{\R}$, be a fixed "box" of geodesics
such that the Liuoville measure $L(Q_{*})$ of the box $Q_{*}$ is
$1$, namely
$$
L(Q_{*})=\log\frac{(c-a)(d-b)}{(d-a)(c-b)}=1.
$$

\begin{lemma}
Let $[h_n]\to [h]$ in $T(\H )$, $\mu_n=\mathcal{EM}([h_n])$ and
$\mu =\mathcal{EM}([h])$. Then $\mu_n\to \mu$ in the following
sense:

\vskip .1 cm

\noindent Let $Q_{*}=[a,b]\times [c,d]$ be a fixed box with
$L(Q_{*})=1$. If $Q$ is any box with $L(Q)=1$, then there exists a
unique $\gamma_Q\in PSL_2(\R )$ such that $\gamma_Q(Q)=Q_{*}$. For
any continuous function $\varphi :\mathcal{G}\to\R$ whose support
is in the fixed box $Q_{*}$, we have that
$$
\mathcal{S}_{\varphi}(\mu_n,\mu )=\sup_{L(Q)=1}\big{|}
\iint_Q\varphi d(\gamma_Q)^{*}(\mu_n )-\iint_Q\varphi
d(\gamma_Q)^{*}(\mu ) \big{|}\to 0
$$
as $n\to\infty$, where the supremum is over all boxes $Q$ with
$L(Q)=1$.
\end{lemma}

\begin{proof}
Assume on the contrary that $\mu_n\nrightarrow\mu$ as
$n\to\infty$. Then there exist a subsequence $Q_{n_k}$ with
$L(Q_{n_k})=1$ and a continuous function $\varphi
:\mathcal{G}\to\R$ whose support is in $Q_{*}$ such that
$$
\mathcal{S}_{\varphi}(\mu_{n_k},\mu )\geq\epsilon_0>0.
$$
We write $n$ instead of $n_k$ for simplicity. Let
$\nu_n=(\gamma_{Q_n})^{*}(\mu_n)$ and
$\zeta_n=(\gamma_{Q_n})^{*}(\mu )$ be the push forwards to $Q_{*}$
of $\mu_n$ and $\mu$ by $\gamma_{Q_n}$. Let $t_{\nu_n}$ and
$t_{\zeta_n}$ be two geodesics in $Q_{*}$ which are either in the
support of $\nu_n$ and $\zeta_n$, or one of them is a subset of a
component of the complement of the support. (Note that for at least
one of the measures $\nu_n$ and $\zeta_n$ there is a geodesic of its
support in $Q_{*}$.)

Let $\delta_n,\delta_n'\in PSL_2(\R )$ be such that
$E^{\nu_n}|_{t_{\nu_n}}=\delta_n\circ
E^{\mu_n}\circ\gamma_{Q_n}^{-1}|_{t_{\nu_n}}=id$ and that
$E^{\zeta_n}|_{t_{\zeta_n}}=\delta_n'\circ
E^{\mu}\circ\gamma_{Q_n}^{-1}|_{t_{\zeta_n}}=id$. Note that
$\|\nu_n\|$ and $\|\zeta_n\|$ are bounded uniformly in $n$ (because
$[h_n]\to [h]$ in the Teichm\"uller metric). It follows that there
exists subsequences of $E^{\nu_n}|_{\hat{\R}}$ and
$E^{\zeta_n}|_{\hat{\R}}$ which converge weakly on $\hat{R}$ to
homeomorphisms $f$ and $g$, respectively (see \cite{Sa2}). Moreover,
there are subsequences of $\nu_n$ and $\zeta_n$ which converge in
the weak* sense to bounded measured laminations $\nu_*$ and
$\zeta_{*}$, respectively. We can assume without loss of generality
that the whole sequence converges. It is clear that
$\nu_{*}\neq\zeta_{*}$ because $\nu_n$ and $\zeta_n$ differ on
$Q_{*}$ by a definite amount (according to the assumption). This
implies that $g\circ f^{-1}$ is not in $PSL_2(\R )$.

On the other hand,
$$
|Belt(DE(E^{\nu_n}|_{\hat{\R}}))-Belt(DE(E^{\zeta_n}|_{\hat{\R}}))|\to
0
$$
uniformly on compact subsets of $\H$ because
$E^{\nu_n}|_{\hat{\R}}=\delta\circ h_n\circ\gamma_n^{-1}$ and
$E^{\zeta_n}|_{\hat{\R}}=\delta'\circ h\circ\gamma_n^{-1}$,
$[h_n]\to [h]$ and by the conformal naturality of the Douady-Earle
extension. This is implies that $Belt(DE(g))=Belt(DE(f))$ which is a
contradiction with the above.
\end{proof}

Let $d([h_1],[h_2])$ denote the Teichm\"uller distance between
$[h_1],[h_2]\in T(\H )$. Assume on the contrary that
$d(E^{(1-t_n)\mu_n}|_{\hat{\R}},E^{(1-t)\mu}|_{\hat{\R}})\nrightarrow
0$ as $n\to\infty$. This is equivalent to the statement that there
exists $z_n\in\H$ such that
$$
\big{|} Belt[DE(E^{(1-t_n)\mu_n}|_{\hat{\R}})](z_n)-
Belt[DE(E^{(1-t)\mu}|_{\hat{\R}})](z_n) \big{|}\geq\epsilon_0>0.
$$
Let $\gamma_n\in PSL_2(\R )$ be such that $\gamma_n(z_n)=i$. There
exist subsequences of $(\gamma_n)_{*}((1-t_n)\mu_n)$ and of
$(\gamma_n)_{*}((1-t)\mu )$ which converge in the weak* sense. We
can assume without loss of generality that both sequences converge
in the weak* sense and the above lemma implies that their limits are
equal. Thus there exist $\delta_n,\delta_n'\in PSL_2(\R )$ such that
$f_n=\delta_n\circ E^{(1-t_n)\mu_n}|_{\hat{\R}}\circ\gamma_n^{-1}$
and $g_n=\delta_n'\circ E^{(1-t)\mu}|_{\hat{\R}}\circ\gamma_n^{-1}$
pointwise converge to the same homeomorphism of $\hat{\R }$. This
implies that $Belt(DE(f_n))$ and $Belt(DE(g_n))$ are close on
compact subsets of $\H$ which gives a contradiction with the
conformal naturality of the Douady-Earle extension and the fact that
$\gamma_n(z_n)=i$ (see the proof of the above lemma for a similar
reasoning).
\end{proof}

The above theorem gives a new proof that the Teichm\"uller space
of an arbitrary Riemann surface is contractible. This was
originally proved by Douady and Earle \cite{DE} (see \cite{EMS}
for contractibility of the subspace of asymptotically conformal
classes).

\begin{corollary}
The Teichm\"uller space $T(\H /G)$ of a hyperbolic surface $\H /G$
is contractible. If $\H /G$ is geometrically infinite, then the
subspace of $T(\H /G)$ consisting of all asymptotically conformal
classes is also contractible.
\end{corollary}

\begin{remark}
The above theorem implies contractibility of any subspace of $T(\H
)$ as long as it is closed under earthquake paths.
\end{remark}

\section{Minimal rate of decrease toward the boundary}

We establish a sufficient condition for a measured lamination $\mu$
on the hyperbolic plane to be trivial in terms of its rate of
decrease when approaching the boundary. It is convenient to work
with the unit disk $\D$ model of the hyperbolic plane, where the
hyperbolic metric is $\rho (z)=\frac{2|dz|}{1-|z|^2}$. Let $I$ be a
closed hyperbolic arc of length $1$ in $\D$. Denote by $\delta (I)$
the minimal Euclidean distance from $I$ to the unit circle
$S^1=\partial\D$. We say that $\mu$ {\it decreases of the order}
$o(\delta )$ and write $\mu =o(\delta )$ if
$$
\mu (I)/\delta(I)\to 0
$$
as $\delta (I)\to 0$ uniformly in $I$.

\begin{proposition}
Let $\mu$ be a measured lamination on $\D$. If $\mu =o(\delta )$
then $\mu$ is trivial, i.e. $\mu =0$.
\end{proposition}

\begin{proof}
Consider the Euclidean circle of radius $1-1/n$ with center $0$. The
hyperbolic length of the circle is $\frac{4\pi (n-1)}{2-1/n}$. The
circle is on the Euclidean distance $1/n$ from $S^1$. The assumption
on $\mu$ implies that the total $\mu$-measure of the circle is
$o(1/n)\cdot \frac{4\pi (n-1)}{2-1/n}\to 0$ as $n\to\infty$. Since
each compact subset of the support of $\mu$ intersects all such
circles for $n$ large enough, it follows that the $\mu$-measure of
each compact set in its support is zero.
\end{proof}


\begin{thebibliography}{Thua}

\vskip .5cm

\bibitem{DE} A. Douady and C. J. Earle, {\it Conformally natural
extension of homeomorphisms of the circle}, Acta Math. 157 (1986),
no. 1-2, 23-48.

\bibitem{EGL} C. Earle, F. Gardiner and N. Lakic, {\it Asymptotic Teichmü\"uller space. II. The metric
structure.}, In the tradition of Ahlfors and Bers, III, 187--219,
Contemp. Math., 355, Amer. Math. Soc., Providence, RI, 2004.

\bibitem{EMS} C. Earle, V. Markovic and D. \v Sari\'c, {\it Barycentric extension and the Bers embedding
for asymptotic Teichm\"uller space}, Complex manifolds and
hyperbolic geometry (Guanajuato, 2001), 87--105, Contemp. Math.,
311, Amer. Math. Soc., Providence, RI, 2002.

\bibitem{EMM} D. Epstein, A. Marden and V. Markovic, {\it Quasiconformal homeomorphisms and the convex hull boundary.},
Ann. of Math. (2) 159 (2004), no. 1, 305--336.

\bibitem{GHL} F. Gardiner, J. Hu and N. Lakic, {\it Earthquake curves}, Complex manifolds and hyperbolic geometry
(Guanajuato, 2001), 141--195, Contemp. Math., 311, Amer. Math. Soc., Providence, RI, 2002.

\bibitem{GS} F. Gardiner and D. Sullivan, {\it Symmetric structures on a closed curve},
Amer. J. Math. 114 (1992), no. 4, 683--736.

\bibitem{Hu} J. Hu, {\it Earthquake measure and cross-ratio
distortion}, In the tradition of Ahlfors and Bers, III, 285--308,
Contemp. Math., 355, Amer. Math. Soc., Providence, RI, 2004.

\bibitem{Ke} S. Kerckhoff, {\it The Nielsen realization problem}, Ann. of Math. (2) 117 (1983), no. 2, 235--265.


\bibitem{Sa1} D. \v Sari\' c, {\it Real and complex earthquakes},
Trans. Amer. Math. Soc. 358 (2006), no. 1, 233--249.

\bibitem{Sa2} D. \v Sari\' c, {\it Bounded earthquakes}, Proc. Amer. Math. Soc. 136 (2008),
no. 3, 889--897 (electronic).


\bibitem{Thu} W. Thurston, {\it Earthquakes in two-dimensional hyperbolic geometry},
Low-dimensional topology and Kleinian groups (Coventry/Durham, 1984), 91--112,
London Math. Soc. Lecture Note Ser., 112, Cambridge Univ. Press, Cambridge, 1986.

\bibitem{Thu2} W. Thurston, {\it Three-dimensional geometry and topology}, Vol. 1.
Edited by Silvio Levy. Princeton Mathematical Series, 35. Princeton University Press, Princeton, NJ, 1997.

\end{thebibliography}
\end{document}